\documentclass[11pt,letterpaper,reqno]{amsart}
\usepackage{tikz}
\usetikzlibrary{positioning, shapes.geometric, arrows.meta, calc, positioning}
\usepackage{amssymb}
\usepackage{amsmath}
\usepackage{amsthm}
\usepackage{amsfonts}
\usepackage{bbm}
\usepackage{enumitem} 
\usepackage{pgfplots}
\pgfplotsset{compat=1.18} 
\usepackage{booktabs}
\usepackage{graphicx}
\usepackage[T1]{fontenc}
\usepackage{doi}
\usepackage{float} 
\usepackage{bookmark}
\usepackage{hyperref}
\hypersetup{pdfstartview={FitH}}
\newcommand{\C}{\mathbb{C}}

\newcommand{\tr}{\operatorname{Tr}}

\newtheorem{thm}{Theorem}[section]
\newtheorem{lem}[thm]{Lemma}
\newtheorem{prop}[thm]{Proposition}
\newtheorem{cor}[thm]{Corollary}

\newtheorem{ques}[thm]{Question}

\newtheorem{conj}[thm]{Conjecture}
\theoremstyle{definition}

\newtheorem{remark}[thm]{Remark}

\numberwithin{equation}{section}

\makeatother

\begin{document}

\title{Generalizing Lee's conjecture on the sum of absolute values of matrices}

\author[Q.~Tang]{Quanyu Tang}
\author[S.~Zhang]{Shu Zhang}

\address{School of Mathematics and Statistics, Xi'an Jiaotong University, Xi'an 710049, P. R. China}
\email{tang\_quanyu@163.com}
\address{School of Mathematics and Statistics, Xi'an Jiaotong University, Xi'an 710049, P. R. China}
\email{2213715573@stu.xjtu.edu.cn}

\subjclass[2020]{Primary 15A60, 47A30.}
\keywords{Lee's conjecture, Schatten norm, Frobenius norm, equiangular system}

\begin{abstract}
Let $\|\!\cdot\!\|_p$ denote the Schatten $p$-norm of matrices and $\|\!\cdot\!\|_F$ the Frobenius norm. 
For a square matrix $X$, let $|X|$ denote its absolute value. 
In 2010, Eun-Young Lee posed the problem of determining the smallest constant $c_p$ such that $\|A+B\|_p \le c_p\|\,|A|+|B|\,\|_p$ for all complex matrices $A,B$. 
The Frobenius case $(p=2)$ conjectured by Lee was proved by Lin and Zhang (2022)~\cite{LinZhang2022} and re-proved by Zhang (2025)~\cite{Zhang2025}. 
In this paper, we extend Lee's conjecture from two matrices to an arbitrary number $m \ge 2$ of complex matrices $A_1,\dots,A_m$, and determine the sharp inequality
$$
   \left\|\sum_{k=1}^{m} A_k\right\|_F
   \le \sqrt{\frac{1+\sqrt{m}}{2}}\;
        \left\|\sum_{k=1}^{m}|A_k|\right\|_F ,
$$
with equality attained by an equiangular rank-one family. 
We further generalize Lee's problem by seeking the smallest constant $c_p(m)$ such that $
   \|\sum_{k=1}^{m} A_k\|_p
   \le c_p(m)\,
        \|\sum_{k=1}^{m}|A_k|\|_p $. It is shown that $c_p(m)\le (\sqrt{m})^{1-1/p}$, and we conjecture a closed-form expression for the optimal value of $c_p(m)$ that recovers all known cases $p=1,2,\infty$.
\end{abstract}

\maketitle

\section{Introduction}
Let $M_n(\mathbb{C}) = \mathbb{C}^{n\times n}$ denote the space of all complex $n\times n$ matrices. For $A\in M_n(\mathbb{C})$, its \emph{absolute value} is defined by $|A|=(A^{*}A)^{1/2}$, where $A^{*}$ denotes the conjugate transpose of $A$.  
The singular values of $A$, arranged in nonincreasing order, are denoted by $s_1(A)\ge s_2(A)\ge\cdots\ge s_n(A)\ge0$. For $p\ge1$, the \emph{Schatten $p$-norm} of $A$ is
\[
   \|A\|_p = \bigl(\tr|A|^p\bigr)^{1/p}
   = \left(\sum_{j=1}^n s_j(A)^p\right)^{1/p}.
\]
The case $p=2$ gives the \emph{Frobenius norm},
\[
   \|A\|_F = \sqrt{\tr|A|^2}
   = \left(\sum_{j=1}^n s_j(A)^2\right)^{1/2}.
\]

For $A,B\in M_n(\mathbb C)$, Lee~\cite{Lee2010} asked for the smallest constant $c_p$ such that
\begin{equation}\label{eq:Lee}
  \|A+B\|_p  \le  c_p \|\,|A|+|B|\,\|_p,
\end{equation}
and conjectured that $c_2=\tfrac{1+\sqrt2}{2}$. The Frobenius case $p=2$ of \eqref{eq:Lee} was later proved in Lin and Zhang~\cite{LinZhang2022}, with an alternative proof via the matrix Cauchy--Schwarz inequality given in Zhang~\cite{Zhang2025}.

In this paper we first establish a sharp generalization of Lee's conjecture for the Frobenius norm involving $m$ matrices.

\begin{thm}\label{thm:sharpF}
For any $m\ge2$ and $A_1,\dots,A_m\in M_n(\C)$,
\begin{equation}\label{eq:sharpF}
\left\|\sum_{k=1}^m A_k\right\|_F
\le
\sqrt{\frac{1+\sqrt m}{2}}\;\left\|\sum_{k=1}^m |A_k|\right\|_F.
\end{equation}
Moreover, the constant $\sqrt{\frac{1+\sqrt m}{2}}$ in inequality \eqref{eq:sharpF} is optimal.
\end{thm}

\begin{remark}
Although we state Theorem~\ref{thm:sharpF} for matrices in $M_n(\mathbb{C})$, the same inequality (with the same proof) still holds for Hilbert--Schmidt operators on a complex Hilbert space and, more generally, for elements of a von Neumann algebra equipped with a regular trace. For simplicity we work in the finite-dimensional matrix setting.
\end{remark}

As a noteworthy special case, taking $m=9$ in Theorem~\ref{thm:sharpF} yields the sharp inequality
\[
  \left\|A_1 + \cdots + A_9\right\|_F
  \le \sqrt{2} \left\| |A_1| + \cdots + |A_9| \right\|_F,
\]
whose constant $\sqrt{2}$ coincides with the sharp constant in the two-matrix operator norm inequality
\[
  \|A+B\|_\infty \le \sqrt{2} \left\||A|+|B|\right\|_\infty.
\]

We then extend Lee's problem to the case of $m$ summands in the Schatten $p$-norm setting.  
\begin{ques}\label{ques:cp}
Let $p\ge1$ and $m\ge2$. 
Find the smallest constant $c_p(m)$ such that
\[
\left\|\sum_{k=1}^m A_k\right\|_p \le c_p(m) \left\|\sum_{k=1}^m |A_k|\right\|_p
\quad\text{for all }A_1,\dots,A_m\in M_n(\C).
\]
\end{ques}
Notice that Theorem~\ref{thm:sharpF} implies $c_2(m)=\sqrt{\frac{1+\sqrt m}{2}}$. For $p=1$ and $p=\infty$, the corresponding optimal constants $c_p(m)$ are also obtained exactly, 
while for general $p\ge1$ we establish in Corollary~\ref{cor:generalize_m_term_boound1} 
a unified upper bound for $c_p(m)$ and further propose in 
Conjecture~\ref{conj:cp} an explicit formula consistent with all known cases.

This paper is organized as follows. 
Section~\ref{sec:main_results1} contains the proof of Theorem~\ref{thm:sharpF}. 
In Section~\ref{sec:remarks1}, we discuss Question~\ref{ques:cp}, 
derive a unified upper bound for $c_p(m)$ in Corollary~\ref{cor:generalize_m_term_boound1}, 
and propose an explicit conjectural formula for $c_p(m)$ in Conjecture~\ref{conj:cp}.

\section{Main results}\label{sec:main_results1}

We will use the following result from \cite[Lemma 2.2]{Zhang2025}.

\begin{lem}[\cite{Zhang2025}]\label{lem:MCS}
Let $X, Y$ be two $n \times n$ positive semidefinite matrices. Assume $Q$ is a contraction of order $n$, i.e., $\|Q\|_\infty\le1$. Then for any $t>0$,
\[
4 \left|\tr(QXY)\right|
\le
t \tr(X^2+Y^2)+\frac1t \tr(XY+YX).
\]
\end{lem}

We are now ready to present the following. Here $\Re z$ denotes the real part of $z \in \mathbb{C}$.

\begin{proof}[Proof of Theorem~\ref{thm:sharpF}]
Write the polar decompositions $A_k=U_k|A_k|$. Expanding the Frobenius norm,
\[
\left\|\sum_{k=1}^m A_k\right\|_F^2
=\sum_{k=1}^m\tr|A_k|^2
+2\sum_{1\le j<k\le m}\Re\tr\left(U_j^\ast U_k |A_k| |A_j|\right).
\]
For each pair $(j,k)$ apply Lemma~\ref{lem:MCS} with $X=|A_j|$, $Y=|A_k|$, $Q=U_j^\ast U_k$;
using $|\Re z|\le |z|$ we get, for every $t>0$,
\[
2\left|\Re\tr(U_j^\ast U_k |A_k| |A_j|)\right|
\le \frac{t}{2} \tr(|A_j|^2+|A_k|^2)+\frac{1}{2t}\tr(|A_j||A_k|+|A_k||A_j|).
\]
Summing over $j<k$ gives
\[
\left\|\sum_{k=1}^m A_k\right\|_F^2
\le \left(1+\frac{m-1}{2}t\right)\sum_{k=1}^m\tr|A_k|^2
+\frac1t\sum_{1\le j<k\le m}\tr(|A_j||A_k|).
\]
Then set \(t=\frac{1}{1+\sqrt{m}}\), we have
\[
\left\|\sum_{k=1}^m A_k\right\|_F^2
\le \frac{1+\sqrt{m}}{2}\left(\sum_{k=1}^m\tr|A_k|^2
+2\sum_{1\le j<k\le m}\tr(|A_j||A_k|)\right)
=\frac{1+\sqrt{m}}{2}\left\|\sum_{k=1}^m |A_k|\right\|_F^2.
\]Taking square roots yields \eqref{eq:sharpF}.

Next we will show that the inequality \eqref{eq:sharpF} is sharp. Fix an undetermined parameter $s\in(0,1)$.  
We claim that there exist unit vectors $x_1,\dots,x_m\in\C^{m}$ such that  
\[
\langle x_j,x_k\rangle=
\begin{cases}
1,& j=k,\\
s,& j\neq k.
\end{cases}
\]
To see why such an equiangular system exists, consider the $m\times m$ Gram matrix
\[
G=(\langle x_j,x_k\rangle)_{j,k=1}^{m}=(1-s)I_m+sJ_m,
\]
where $J_m$ denotes the all–ones matrix. The matrix $G$ is positive semidefinite, so $G=X^*X$ for some $X$, and with $x_k$ the $k$th column of $X$ we obtain our unit equiangular system.

Fix a unit vector $u\in\C^m$ and set
\[
A_k:=u\,x_k^{*}\qquad(k=1,\dots,m).
\]
Then
\[
A_k^{*}A_k=x_k\,x_k^{*},
\quad\text{so}\quad
|A_k|=(A_k^{*}A_k)^{1/2}=(x_k x_k^{*})^{1/2}=x_k x_k^{*}.
\]
Here $x_k x_k^{*}$ is the orthogonal projection onto $\operatorname{span}\{x_k\}$ (it is also positive
semidefinite), hence equals its own positive square root. Now we compute both sides:
\[
\left\|\sum_{k=1}^m A_k\right\|_F^2
=\left\|\,u\sum_{k=1}^mx_k^{*}\,\right\|_F^2
=\left\|\sum_{k=1}^mx_k\right\|_2^2
= m + 2\sum_{1\le j<k\le m}\Re\langle x_j,x_k\rangle
= m + m(m-1)s,
\]
\[
\left\|\sum_{k=1}^m |A_k|\right\|_F^2
=\left\|\sum_{k=1}^m x_k x_k^{*}\right\|_F^2
= \sum_{j,k=1}^m \tr(x_j x_j^{*} x_k x_k^{*})
= m + 2\sum_{1\le j<k\le m} |\langle x_j,x_k\rangle|^{2}
= m + m(m-1)s^{2}.
\]Therefore
\[
\frac{\left\|\sum_{k=1}^m A_k\right\|_F^{2}}{\left\|\sum_{k=1}^m |A_k|\right\|_F^{2}}
=\frac{1+(m-1)s}{1+(m-1)s^{2}}
=: f(s).
\]Optimize $f(s)$ on $(0,1)$:
\[
f'(s)=0
\ \Longleftrightarrow\
(m-1)s^{2}+2s-1=0
\ \Longleftrightarrow\
s^{\ast}=\frac{\sqrt m-1}{m-1}=\frac{1}{1+\sqrt m}.
\]
Let $s=s^{\ast}\in(0,1)$, substituting back,
\[
f(s^{\ast})=\frac{1+\sqrt m}{2}
\quad\Longrightarrow\quad
\frac{\left\|\sum_{k=1}^m A_k\right\|_F}{\left\|\sum_{k=1}^m |A_k|\right\|_F}
 =
\sqrt{\frac{1+\sqrt m}{2}}.
\]
This shows that the inequality \eqref{eq:sharpF} is sharp. The proof is complete.\end{proof}

\begin{remark}\label{rem:m2-equality}
The extremal rank-one equiangular construction above reduces, when $m=2$, to the known equality case in Lee's conjecture; see \cite{LinZhang2022}. 
\end{remark}



\section{Concluding Remarks}\label{sec:remarks1}
We now turn to the generalized version of Lee's problem, 
that is, Question~\ref{ques:cp}.  
Supported by numerical experiments, 
we are led to the following conjectural formula for $c_p(m)$.

\begin{conj}\label{conj:cp}Let $p>1$ and $m\ge2$. Let $c_p(m)$ be the smallest number such that
\[
\left\|\sum_{k=1}^m A_k\right\|_p \le c_p(m) \left\|\sum_{k=1}^m |A_k|\right\|_p
\quad\text{for all }A_1,\dots,A_m\in M_n(\C).
\]
Then\footnote{
The limiting case $p=1$ is discussed separately in 
Remark~\ref{remark:p=1,2,infinity}.}
\begin{equation}\label{eq:cp-conj}
c_p(m)
=\frac{\sqrt{x_{p,m}\left(x_{p,m}+m-1\right)}}{\left(x_{p,m}^{\,p}+m-1\right)^{1/p}},
\qquad
x_{p,m}>0\text{ solves }x^p-2x-(m-1)=0.
\end{equation}
\end{conj}

\begin{proof}[A heuristic derivation of Conjecture~\ref{conj:cp}]We again employ the construction that realizes equality in the proof of 
Theorem~\ref{thm:sharpF}. Keep the rank-one equiangular family $A_k=u\,x_k^\ast$, $\langle x_j,x_k\rangle=s\in(0,1)$.
As $\sum_k A_k=u\sum_{k}x_k^{*}$ is rank one, its Schatten $p$-norm equals its unique singular value:
\[
\left\|\sum_{k=1}^m A_k\right\|_p=\left\|\sum_{k=1}^m x_k\right\|_2=\sqrt{m+m(m-1)s}.
\]
The positive operator $\sum_k|A_k|=\sum_k x_kx_k^\ast$ has eigenvalues
$1+(m-1)s$ (multiplicity $1$) and $1-s$ (multiplicity $m-1$), so
\[
\left\|\sum_{k=1}^m |A_k|\right\|_p=\Big(\left(1+(m-1)s\right)^p+(m-1)(1-s)^p\Big)^{1/p}.
\]
Maximizing the ratio over $s\in(0,1)$ is equivalent to introducing
\(
y=\frac{1+(m-1)s}{1-s}\in (1,\infty)
\)
and optimizing
\[
R_p(y)=
\frac{\sqrt{y(y+m-1)}}{\left(y^p+m-1\right)^{1/p}}.
\]A logarithmic derivative shows the maximizer $y$ satisfies
$y^p=2y+(m-1)$, and substituting $x_{p,m}:=y$ gives \eqref{eq:cp-conj}.\end{proof}

In~\cite{Lee2010}, Lee mentioned that ``one may easily state a version of 
\cite[Corollaries~2.2 and~2.2a]{Lee2010} for finite families of operators'', 
but did not explicitly write down the corresponding $p$-norm formulation. For completeness, we state and prove such a version here. 
The argument follows the same structure as Lee's original proof, with minor modifications to handle $m$ summands. 

We recall that a function $f:[0,\infty)\to[0,\infty)$ is said to be \emph{geometrically convex} if $f\left(\sqrt{ab}\right) \le \sqrt{f(a)f(b)}$ for all $a,b>0$. We also recall that \emph{symmetric norms} refer to \emph{unitarily invariant norms}, namely those matrix norms satisfying $\|UAV\| = \|A\|$ for all unitary matrices $U,V$. In particular, the Schatten $p$-norms $(1\le p\le\infty)$ are unitarily invariant.

\begin{prop}\label{prop:generalize-m-term1}
Let $f:[0,\infty)\to[0,\infty)$ be concave and geometrically convex. Then for any matrices $A_1,\dots,A_m \in M_n(\C)$ and any Schatten norm $\|\cdot\|_q$ $(1\le q\le\infty)$,
\[
\left\|f \left(|A_1+\cdots+A_m|\right)\right\|_q
\le \left(\sqrt{m}\right)^{1-1/q} 
\left\|f(|A_1|)+\cdots+f(|A_m|)\right\|_q .
\]
\end{prop}

\begin{proof}Let
\[
S:=\sum_{i=1}^m A_i,\qquad
P:=\sum_{i=1}^m |A_i^*|,\qquad
Q:=\sum_{i=1}^m |A_i|.
\]By \cite[Corollary~1.3.7]{Bhatia09}, for each $i$, the block matrix
$\begin{pmatrix}|A_i^*|&A_i\\ A_i^*&|A_i|\end{pmatrix}$ is positive semidefinite.
Summing over $i$ yields \(\begin{pmatrix} P & S\\ S^* & Q\end{pmatrix}\ \ge\ 0\). By \cite[Proposition~1.3.2]{Bhatia09}, there exists a contraction $K$ (i.e.\ $\|K\|_\infty\le1$)
such that $S=P^{1/2}\,K\,Q^{1/2}$. Let $s_j(\cdot)$ denote the singular values and $\lambda_j(\cdot)$ the eigenvalues of an operator (both in nonincreasing order). In the same way as in \cite[p.~208]{HJ91}, for every $k\ge1$, we have
\[
\prod_{j=1}^k \lambda_j(|S|) = \prod_{j=1}^k s_j(S) \le
\prod_{j=1}^k s_j(P^{1/2}) s_j(Q^{1/2})
=\prod_{j=1}^k \lambda_j^{1/2}(P)\lambda_j^{1/2}(Q).
\]Since $f$ is geometrically convex and increasing, we observe that $\phi(t)=f(e^t)$ is convex and increasing. The above weak log-majorisation then yields the following weak majorisation:
for all $k=1,2,\ldots,$
\[
\sum_{j=1}^k \lambda_j\left(f(|S|)\right)
\le
\sum_{j=1}^k f \left[\lambda_j^{1/2}(P)\lambda_j^{1/2}(Q)\right]
\le
\sum_{j=1}^k
\lambda_j^{1/2}\left(f(P)\right)\lambda_j^{1/2}\left(f(Q)\right).
\]Let $X$ be the diagonal matrix whose diagonal entries are
\[
\lambda_1^{1/2}\left(f(P)\right),\,
\lambda_2^{1/2}\left(f(P)\right),\,
\ldots,
\]and let $Y$ be the diagonal matrix with diagonal entries
\[
\lambda_1^{1/2}\left(f(Q)\right),\,
\lambda_2^{1/2}\left(f(Q)\right),\,
\ldots.
\]The above weak majorisation implies that
$\|f(|S|)\|\le \|XY\|$ for all symmetric norms. By the Cauchy–Schwarz inequality for symmetric norms,
\[
\|XY\|\le\|X^*X\|^{1/2}\|Y^*Y\|^{1/2},
\]
and hence
\begin{equation}\label{eq:geomean}
\|f(|S|)\|\le
\|f(P)\|^{1/2}\|f(Q)\|^{1/2}.
\end{equation}Now, since $f:[0,\infty)\to[0,\infty)$ is concave, we apply the Bourin--Uchiyama
subadditivity theorem (see \cite[Theorem~1.1]{BU07} and the remarks in
\cite[p.~516]{BU07}). For all symmetric norms, we have
\begin{equation}\label{eq:BU}
\|f(P)\|_q \le \left\|\sum_{i=1}^m f(|A_i^*|)\right\|_q,
\qquad
\|f(Q)\|_q \le \left\|\sum_{i=1}^m f(|A_i|)\right\|_q .
\end{equation}
Combining \eqref{eq:geomean} and \eqref{eq:BU} yields
\begin{equation}\label{eq:afterBU}
\|f(|S|)\|_q \ \le\
\left\|\sum_{i=1}^m f(|A_i^*|)\right\|_q^{1/2}
\left\|\sum_{i=1}^m f(|A_i|)\right\|_q^{1/2}.
\end{equation}By \cite[Lemma~2.2]{CondeMoslehian2016}, for arbitrary positive semidefinite $X_1,\dots,X_m$, we have
\begin{equation}\label{eq:triangle1}
\left\|\sum_{i=1}^m X_i\right\|_q  \le
\sum_{i=1}^m \|X_i\|_q \ \le\ m^{1-1/q}\left(\sum_{i=1}^m \|X_i\|_q^q\right)^{1/q}
= m^{1-1/q}\left\|\bigoplus_{i=1}^m X_i\right\|_q .
\end{equation}
and
\begin{equation}\label{eq:triangle2}
\left\|\bigoplus_{i=1}^m X_i\right\|_q
 \le \left\|\sum_{i=1}^m X_i\right\|_q .
\end{equation}
Since $|A_i^*|$ and $|A_i|$ have identical singular values, combining inequalities~\eqref{eq:triangle1} and \eqref{eq:triangle2} with
\eqref{eq:afterBU} yields
\[
\begin{aligned}
\|f(|S|)\|_q
&\le \left(m^{1-1/q}\left\|\bigoplus_{i=1}^m f(|A_i|)\right\|_q\right)^{1/2}
     \left\|\sum_{i=1}^m f(|A_i|)\right\|_q^{1/2}
\\
&\le m^{(1-1/q)/2}
   \left\|\sum_{i=1}^m f(|A_i|)\right\|_q^{1/2}
   \left\|\sum_{i=1}^m f(|A_i|)\right\|_q^{1/2} 
\\
&= \left(\sqrt{m}\right)^{1-1/q} \left\|\sum_{i=1}^m f(|A_i|)\right\|_q.
\end{aligned}
\]This completes the proof.
\end{proof}

As a direct corollary of Proposition~\ref{prop:generalize-m-term1}, we obtain the following unified upper bound for $c_p(m)$.
\begin{cor}\label{cor:generalize_m_term_boound1}
For any $p \geq 1$ and $m\ge2$,
\begin{equation}
c_p(m)\le
\left(\sqrt{m}\right)^{1-1/p}.
\end{equation}    
\end{cor}

\begin{remark}\label{remark:p=1,2,infinity}
The conjectural formula \eqref{eq:cp-conj} agrees with the exact values at $p=1,2,\infty$:
\begin{itemize}[leftmargin=1.2em]
\item $p=1$: letting $p\to 1^{+}$ in \eqref{eq:cp-conj} forces $x_{p,m}\to\infty$, so $c_p(m)\to 1$; in fact
$c_1(m)=1$ exactly, since Corollary~\ref{cor:generalize_m_term_boound1} gives $c_1(m)\le1$.
\item $p=2$: $x_{2,m}=1+\sqrt m$ and hence 
$c_2(m)=\sqrt{\tfrac{1+\sqrt m}{2}}$, in agreement with Theorem~\ref{thm:sharpF}.
\item $p=\infty$: as $p\to\infty$ we have $x_{p,m}\to1$ and
$(x_{p,m}^{p}+m-1)^{1/p}\to1$, hence \eqref{eq:cp-conj} yields
$c_{p}(m)\to\sqrt m$; in fact
$c_{\infty}(m)=\sqrt{m}$ exactly, since Corollary~\ref{cor:generalize_m_term_boound1} implies $c_{\infty}(m)\le \sqrt{m}$.
\end{itemize}
\end{remark}

\section*{Acknowledgements}
We are deeply grateful to Prof.~Minghua Lin for bringing this problem to our attention and for many insightful discussions. We also thank Teng Zhang for helpful comments. Finally, we thank the anonymous referee for the careful reading and constructive suggestions, which greatly improved the clarity of the paper.

\end{document}